\documentclass[letter, reqno, 12pt]{amsart}
\usepackage{amsmath,amsthm,amssymb}
\usepackage{comment}
\usepackage[margin=1in]{geometry}
\usepackage{float}
\usepackage[citestyle=alphabetic,bibstyle=alphabetic,maxbibnames=5]{biblatex}
\usepackage{graphicx}

\newcommand{\R}{\mathbb{R}}
\newcommand{\E}{\mathbb{E}}

\newcommand{\Ai}{\text{Ai}}

\newtheorem{theorem}{Theorem}

\newtheorem{proposition}{Proposition}

\theoremstyle{definition}

\theoremstyle{remark}

\theoremstyle{definition}

\theoremstyle{definition}

\title{Bounded Bessel Processes and Ferrari-Spohn Diffusions}
\author{Matthew Lerner-Brecher}
\date{}

\begin{document} 
\begin{abstract} We introduce a new diffusion process which arises as the $n\to\infty$ limit of a Bessel process of dimension $d \ge 2$ conditioned upon remaining bounded below one until time $n$. In addition to being interesting in its own right, we argue that the resulting diffusion process is a natural hard edge counterpart to the Ferrari-Spohn diffusion of \cite{FS}. In particular, we show that the generator of our new diffusion has the same relation to the Sturm-Liouville problem for the Bessel operator that the Ferrari-Spohn diffusion does to the corresponding problem for the Airy operator.
\end{abstract}
\maketitle

\section{Introduction}
The Ferrari-Spohn diffusion is a diffusion process on the positive real line with infinitesimal generator
\[\frac{1}{2}\frac{d^2}{dx^2} + \left(\frac{d}{dx} 
\log\Ai(x - \omega_1)\right)\frac{d}{dx},\]
where $\Ai(x)$ denotes the Airy function and $-\omega_i$ is the $i$-th largest real root of $\Ai(x)$. While Ferrari-Spohn diffusions have far-reaching connections to the Ising and SOS models \cite{ISV, IVW, IOSV, FSh}, the original motivation for studying them came from an approximation to multilayer polynuclear growth given by Brownian bridges \cite{FS}. In particular, the authors of \cite{FS} looked at a collection of $n$ Brownian bridges $B_i:[0,T]\to\R$, $B_i(0) = B_i(T) = -i$, $1 \le i \le n$ conditioned on non-intersection. Under proper re-scaling as $n\to\infty$, the top-most Brownian bridges asymptotically approach a semicircle and their behavior at a fixed time is given by the Airy kernel. The idea of \cite{FS} was to observe how this process changes if one replaces the lower $n-1$ probabilistic curves with a single deterministic one, that is, a Brownian bridge conditioned to stay above a semicircle. While many large scale properties remain unchanged, at the local level this new process instead evolves according to the Ferrari-Spohn diffusion. This connection between Ferrari-Spohn diffusions and the Airy kernel has since been expounded on even further as recent papers \cite{FSh, DS} have shown that if one conditions $n$ Ferrari-Spohn diffusions on non-intersection (the so-called Dyson Ferrari-Spohn diffusion), then the $n \to \infty$ behavior of the top lines is given by the Airy line ensemble of \cite{CH}.
\par 
In this paper, we introduce a new diffusion process, which we argue is the natural analogy of the Ferrari-Spohn diffusion for the Bessel kernel. Like the Airy kernel, the Bessel kernel is a limiting correlation kernel that often emerges when studying edge limits of determinantal point processes. Typically, these kernels arise under similar conditions with the caveat that one expects the Airy kernel where the support is unbounded and the Bessel kernel where the support is constrained. A classical example of this is given by the singular values of a random square Gaussian matrix. As the dimension of the matrix goes to infinity, the limiting correlation kernel near the largest eigenvalues is the Airy kernel, whereas near the smallest eigenvalues, which are bounded below by 0, the same limit reveals the Bessel kernel (see e.g. \cite[Ch. 7]{For}).
\par 
To determine such a diffusion we looked towards a non-intersecting line ensemble presented in \cite{KMW}. Here, one studies $n$ Bessel bridges $Y_i:[0,T]\to\R_{\ge 0}$ of dimension $d$ satisfying $Y_i(0)=a, Y_i(T) = 0, 1 \le i \le n$ and again conditioned on non-intersection. For integer $d$, Bessel processes are equal to the magnitude of a $d$ dimensional Brownian motion, so this is a natural extension of the Brownian bridge model. As before, under appropriate re-scaling as $n\to\infty$, the bottom and top curves asymptotically approach deterministic curves. For some critical time $t^*$, the fixed time $t$ behavior of the bottom curves is described by the Airy kernel if $t < t^*$ and the Bessel kernel if $t > t^*$. Near the top curve, the fixed time behavior is given by the Airy kernel for any $t \in (0,T)$. While they do not prove this formally, the authors of \cite{FS} explain that local convergence to the Ferrari-Spohn diffusion should hold for a Brownian bridge conditioned to stay above any concave curve $g:[0,T]\to\R$. It is then reasonable to expect that if one performs the same procedure of conditioning a Bessel bridge above (resp. below for $t < t^*$) the limiting deterministic curve near the top (resp. bottom) of the ensemble, the Ferrari-Spohn diffusion will emerge.
\par 
Unfortunately, for $t > t^*$, the bottom deterministic curve becomes a flat line at zero, so it is meaningless to constrain a Bessel process below it. However, since the average distance of the bottom-most Bessel bridge from zero is constant and of order $n^{-1}$, a natural alternative is to condition a Bessel process on remaining below $n^{-1}$ over some fixed time interval. Scaling this up, one gets a Bessel process conditioned on staying below 1 until time of order $n$. The main result of this paper is that this remarkably simple process converges to a diffusion whose generator bears a strong connection to that of the Ferrari-Spohn diffusion.
\begin{theorem} 
\label{main_thm}
Let $Y$ be a Bessel process of dimension $d \ge 2$ with $Y_0 \in (0,1)$ fixed. Set $\tau = \inf \{t > 0 : Y_t = 1\}$ and define $X^{(n)}$ to be a stochastic process distributed according to $Y | \tau > n$ for $n \ge 0$. Then, in the sense of finite dimensional distributions, as $n\to\infty$ the process $X^{(n)}$ converges to a Feller process $X$ with state space $[0,1]$ and generator
\begin{equation} 
    \label{eq:bdd_Bessel_generator}
\mathcal{L} = \frac{1}{2}\frac{d^2}{dx^2} + \left(\frac{d}{dx} \log (x^{1/2} J_{\alpha}(j_{1,\alpha}x))\right)\frac{d}{dx}.
\end{equation}
Here $\alpha = \frac{d-2}{2}$, $J_{\alpha}(x)$ denotes the Bessel function of the first kind of order $\alpha$, and $0 < j_{1, \alpha} < j_{2,\alpha} < \cdots$ are the positive real zeros of $J_{\alpha}(x)$.
\end{theorem} 
The function $J_{\alpha}(j_{1,\alpha}x)$ is a natural counterpart to $\Ai(x - \omega_1)$. Indeed, these are solutions to the eigenvalue (equiv. Sturm-Liouville) problem $Dy = \lambda y$ corresponding to differential operators defining the Bessel, Airy functions respectively. The presence of $x^{1/2}$ can be explained by orthogonality relations satisfied by solutions to these problems. The functions $\{J_{\alpha}(j_{i,\alpha}x)\}_{i=1}^{\infty}$ are orthogonal with respect to the measure $xdx$ on $[0,1]$, whereas $\{\Ai(x-\omega_i)\}_{i=1}^{\infty}$ are orthogonal with respect to $dx$ on $[0,\infty)$. Interestingly, these same functions appear when studying the classical random matrix ensembles at zero temperature. At the soft edge \cite{AHV, GK}, the limiting behavior of the eigenvalues is described by Gaussians whose covariances are expressible in terms of $\Ai(x - \omega_i)$. At the hard edge \cite{And21, LB}, one instead obtains functions of $J_{\alpha}(j_{i,\alpha}x)$ again re-scaled to have orthogonality measure $dx$.
\par 
The proof of our main result will be spread across two sections. In Section 2, we will derive explicit forms for the limiting distribution of $X^{(n)}_t$ as $n\to\infty$. In Section 3, we will then prove that this limiting distribution defines a Feller semigroup with the desired generator. Between this and a brief argument showing $X^{(n)}$ becomes Markovian in the limit, the result will follow. In both sections, our primary tools will be Dynkin's theorem and a Fourier-like expansion involving the functions $\{J_{\alpha}(j_{i,\alpha}x)\}_{i=1}^{\infty}$. To our knowledge, such expansions are the only way of gaining access to an explicit formula for the joint density of a Bessel process and the time it takes it to reach 1.

\section{Transition Probabilities}
We begin by introducing Bessel processes, Bessel functions, and the aforementioned Fourier-like expansion more formally. The Bessel process of dimension $d$ is a diffusion on $[0,\infty)$ given by the solution to the SDE 
\begin{equation} 
\label{eq:Bessel_SDE} dY_t = dB_t + \frac{d-1}{2Y_t}dt,
\end{equation}
where $B$ is a one-dimensional Brownian motion. In the relevant case $d \ge 2$, $Y_t$ is Feller and has generator
\begin{equation} \label{eq:L0}
\mathcal{L}_0 := \frac{1}{2}\frac{d^2}{dx^2} + \frac{d-1}{2x}\frac{d}{dx}
\end{equation}
whose domain contains all bounded $C^2[0,\infty)$ functions. For facts about Bessel processes pertaining to the generator see \cite{BS}; for all other properties see \cite[Ch. XI]{RY}.
\par
The Bessel function of the first kind $J_{\alpha}(x)$ is a solution to the differential equation 
\begin{equation} 
\label{eq:J_de}
\frac{d^2y}{dx^2} + \frac{1}{x} \frac{dy}{dx} + \left(1 - \frac{\alpha^2}{x^2}\right)y = 0
\end{equation}
that satisfies
\begin{equation} \label{eq:J_Taylor_series} J_{\alpha}(x) = \sum_{k=0}^{\infty} (-1)^k \frac{(x/2)^{2k+\alpha}}{k!\Gamma(k+\alpha+1)}.
\end{equation}
For $\alpha > -1$, $J_{\alpha}(x)$ has an infinite number of positive real zeros $j_{1,\alpha} < j_{2,\alpha} < \cdots$ all of which are simple. These zeros play an important role as one can obtain Fourier-like expansions of a large class of functions in terms of the basis $\{J_{\alpha}(j_{i,\alpha}x)\}_{i=1}^{\infty}$. Indeed, it is a classical result of Sturm-Liouville theory that for any $f(x)$ with $x^{1/2}f(x)$ integrable on $(0,1)$ and for any $\alpha > -\frac{1}{2}$ the equality
\begin{equation} 
\label{eq:Fourier_Bessel_Expansion}
f(x) = 2\sum_{i=1}^{\infty} \frac{J_{\alpha}(j_{i,\alpha}x)}{J_{\alpha+1}(j_{i,\alpha})^2} \int_0^1 f(y)J_{\alpha}(j_{i,\alpha}y)ydy
\end{equation}
holds under the same convergence criteria as the usual Fourier expansion. This expansion is often referred to as the Fourier-Bessel series for $f(x)$ and it will hold the key to proving the following proposition regarding the limiting transition probabilities of a bounded Bessel processes. For a comprehensive introduction to Bessel functions, their zeros, and Fourier-Bessel series we refer the reader to \cite{Wat}.

\begin{proposition} \label{prop:transition_probabilities} Let $Y$ be stochastic process and let $(\mathbb{P}^x)_{x \in (0,1)}$ be a collection of probability measures such that under $\mathbb{P}^x$, $Y$ is a Bessel process of dimension $d$ with $Y_0 = x$ almost surely. Additionally, define $X^{(n)}, \tau, \alpha$ as in Theorem \ref{main_thm}. Then, for any $t > 0$,
\begin{equation} \label{eq:transition_prob_bdd_bessel} \lim_{n\to\infty} \mathbb{P}^x(X^{(n)}_t \in dy) = \mathbb{P}^x(Y_t \in dy, \tau > t) \exp\left(\frac{j_{1,\alpha}^2}{2}t\right)\frac{y^{-\alpha}J_{\alpha}(j_{1,\alpha} y)}{x^{-\alpha}J_{\alpha}(j_{1,\alpha} x)},
\end{equation}
where the right-hand side may be further decomposed via the identity
\begin{equation} 
\label{eq:transition_prob_killed_Bessel}
\mathbb{P}^x(Y_t \in dy, \tau > t) = 2\frac{y^{\alpha+1}}{x^{\alpha}}\sum_{i=1}^{\infty}  \frac{J_{\alpha}(j_{i,\alpha} x)J_{\alpha}(j_{i,\alpha} y)}{J_{\alpha+1}(j_{i,\alpha})^2} \exp\left(-\frac{j_{i,\alpha}^2}{2}t\right)dy.
\end{equation}
\end{proposition}

\begin{proof}
We begin by proving the latter identity. Let $h_i(x) = x^{-\alpha} J_{\alpha}(j_{i,\alpha}x)$. By  \eqref{eq:J_Taylor_series}, $h_i \in C^2[0,\infty)$ and thus we may apply Dynkin's formula using the generator $\mathcal{L}_0$ for $Y_t$:
\[
\E^x[h_i(Y_{t \wedge \tau})] = h_i(x) + \E^x\left[\int_0^{t \wedge \tau} (\mathcal{L}_0h_i)(Y_s)ds\right].
\]
From the differential equation \eqref{eq:J_de} defining $J_{\alpha}(x)$, we obtain $\mathcal{L}_0 h_i = -\frac{j_{i,\alpha}^2}{2}h_i$. Hence
\begin{align*} 
\E^x[h_i(Y_{t \wedge \tau})] &= h_i(x) - \frac{j_{i,\alpha}^2}{2} \E^x\left[\int_0^{t \wedge \tau} h_i(Y_s)ds\right] \\
&= h_i(x) - \frac{j_{i,\alpha}^2}{2} \E^x\left[\int_0^{t} h_i(Y_{s \wedge \tau})ds\right] \\
&= h_i(x) - \frac{j_{i,\alpha}^2}{2} \int_0^{t} \E^x\left[h_i(Y_{s \wedge \tau})\right]ds,
\end{align*}
where we use $h_i(1) = 0$ in the second line and swap the the integral using the boundedness of $h_i(x)$ on $[0,1]$ in the third line. Solving this yields
\begin{equation} \label{eq:Bessel-Fourier_coeffs} \int_0^1 h_i(y)\mathbb{P}^x(Y_t \in dy, \tau > t) = \E^x[h_i(Y_{t\wedge \tau})] = \exp\left(-\frac{j_{i,\alpha}^2}{2}t\right)h_i(x).
\end{equation}
Finally, by applying the Fourier-Bessel expansion \eqref{eq:Fourier_Bessel_Expansion} to 
\[f(y)dy = y^{-\alpha-1} \mathbb{P}^x(Y_t \in dy, \tau > t)\]
we obtain \eqref{eq:transition_prob_killed_Bessel}. The function $y^{1/2} f(y)$ can be seen to be integrable via the upper bound 
\[ y^{-\alpha-\frac{1}{2}} \mathbb{P}^x(Y_t \in dy) = \frac{y^{\frac{1}{2}}}{tx^{\alpha}}\exp\left(-\frac{x^2 + y^2}{2t}\right)I_{\alpha}\left(\frac{xy}{t}\right)dy, \]
holding for $\alpha \ge 0$. Here $I_{\alpha}(x) = i^{-\alpha}J_{\alpha}(ix)$ is the \textit{modified Bessel function of the first kind}. Additionally, the right-hand side of \eqref{eq:transition_prob_killed_Bessel} converges absolutely as $j_{i,\alpha}$ grows on the order of $i$, $J_{\alpha}(x)$ is upper bounded by a polynomial in $x$, and $|J_{\alpha+1}(j_{i,\alpha})|$ can be satisfactorily lower bounded via the identity
\[\frac{1}{2}J_{\alpha+1}(j_{i,\alpha})^2 = \int_0^1 J_{\alpha}(j_{i,\alpha}x)^2 dx\]
following from the Fourier-Bessel expansion \eqref{eq:Fourier_Bessel_Expansion}.
\par
To prove \eqref{eq:transition_prob_bdd_bessel}, we begin with Bayes' rule:
\[\mathbb{P}^x(Y_t \in dy | \tau > n) = \frac{\mathbb{P}^x(Y_t \in dy) \mathbb{P}^x(\tau > n | Y_t = y)}{\mathbb{P}^x(\tau > n)}.\]
By the Strong Markov property
\[ \mathbb{P}^x(\tau > n | Y_t = y) = \mathbb{P}^x(\tau >t|Y_t=y)\mathbb{P}^y(\tau > n-t) \]
whenever $n>t$. Plugging this back into Bayes' rule yields
\begin{equation} \label{eq:Bayes_rule}
    \mathbb{P}^x(Y_t \in dy | \tau > n) = \mathbb{P}^x(\tau > t, Y_t \in dy)\frac{\mathbb{P}^y(\tau > n-t)}{\mathbb{P}^x(\tau > n)}
\end{equation}
To compute this we use the following identity, which may be found in \cite{BS} or can be derived from \eqref{eq:transition_prob_killed_Bessel} by integrating over $y$ from 0 to 1:
\begin{equation}
    \label{eq:hitting_time_dist}
    \mathbb{P}^x(\tau > t) = 2x^{-\alpha}\sum_{k=1}^{\infty} \frac{J_{\alpha}(j_{k,\alpha}x)}{j_{k,\alpha}J_{\alpha+1}(j_{k,\alpha})} \exp\left(-\frac{j_{k,\alpha}^2}{2} t\right).
\end{equation}
Indeed,
\[\lim_{n\to\infty} \frac{\mathbb{P}^y(\tau > n-t)}{\mathbb{P}^x(\tau > n)} = \lim_{n\to\infty} \displaystyle\frac{2y^{-\alpha}\sum_{k=1}^{\infty} \frac{J_{\alpha}(j_{k,\alpha}y)}{j_{k,\alpha}J_{\alpha+1}(j_{k,\alpha})} \exp\left(-\frac{j_{k,\alpha}^2}{2}(n-t)\right)}{2x^{-\alpha}\sum_{k=1}^{\infty} \frac{J_{\alpha}(j_{k,\alpha}x)}{j_{k,\alpha}J_{\alpha+1}(j_{k,\alpha})} \exp\left(-\frac{j_{k,\alpha}^2}{2}n\right)} \]
In the limit $n \to\infty$, the $k = 1$ term will dominate the numerator and denominator yielding:
\[\lim_{n\to\infty} \frac{\mathbb{P}^y(\tau > n-t)}{\mathbb{P}^x(\tau > n)} = \frac{y^{-\alpha}J_{\alpha}(j_{1,\alpha}y)}{x^{-\alpha}J_{\alpha}(j_{1,\alpha}x)} \exp\left(\frac{j_{1,\alpha}^2}{2}t\right)
\]
Combining this with our application of Bayes' rule \eqref{eq:Bayes_rule} completes the proof.
\end{proof}

\section{Feller Property and Generator}
Define $h_i(x), \tau$ as in Proposition \ref{prop:transition_probabilities} and set $Z_t$ to be the process $Y_t$ killed at time $\tau$. The density of the limiting transition probabilities \eqref{eq:transition_prob_bdd_bessel} has a simple interpretation as the Doob's $h$-transform of $Z_t$ by $h=h_1$ rescaled by an exponential factor. This factor is present as $h_1$ is an eigenvalue of $\mathcal{L}_0$ instead being harmonic with respect to it. As the Doob's $h$-transform of a generator $\mathcal{G}$ is given by $f(x) \mapsto h(x)^{-1}(\mathcal{G}hf)(x)$, one would then expect that \eqref{eq:transition_prob_bdd_bessel} defines a semigroup with generator 
\begin{equation} \label{eq:Doob_h_transform}
(\mathcal{L}f)(x) = \frac{1}{h(x)} \left(\mathcal{L}_0 + \frac{j_{1,\alpha}^2}{2}\right) h(x)f(x)
\end{equation}
and indeed this $\mathcal{L}$ equals the one from the statement of Theorem \ref{main_thm}. The majority of the following proof will devoted to showing this idea more formally and confirming that \eqref{eq:transition_prob_bdd_bessel} does indeed define a Feller semigroup.
\begin{proof}[Proof of Theorem \ref{main_thm}]
Set $Q_t(x,y)dy$ to be the right-hand side of \eqref{eq:transition_prob_bdd_bessel} and 
\[R_t(x,y)dy = \mathbb{P}^x(Y_t \in dy, \tau > t).\]
Our argument will be separated into three parts: First, we will check that $Q_t$ indeed defines a Feller semigroup; following that, we will show $Q_t$ has generator $\mathcal{L}$; and then finally we complete the proof by showing convergence in the sense of finite dimensional distributions.
\par 
Before we even can confirm the Feller properties, however, we need to define $Q_t(x,y)$ for $x = 0,1$. Combining our two equations from Proposition \ref{prop:transition_probabilities}, we get
\begin{equation} 
\label{eq:Q_full_transition_prob}
Q_t(x,y) = 2yJ_{\alpha}(j_{1,\alpha}y) \sum_{i=1}^{\infty} \frac{J_{\alpha}(j_{i,\alpha}y)}{J_{\alpha+1}(j_{i,\alpha})^2} \frac{h_{i}(x)}{h_1(x)}\exp\left(\frac{j_{1,\alpha}^2 - j_{i,\alpha}^2}{2}t\right).
\end{equation}
Recall that $h_i(x)$ is continuous on the interval $[0,1]$ and polynomially bounded in $i$. As $h_i(0) \neq 0$ and 1 is a simple root of $h_i$ for all $i$, $\frac{h_i(x)}{h_1(x)}$ can be extended to a continuous function on $[0,1]$ bounded polynomially in $i$. Thus we can extend $Q_t(x,y)$ to $x = 0,1$ by continuity. In fact, the absolute convergence of \eqref{eq:Q_full_transition_prob} shows the map $(t,x,y)\mapsto Q_t(x,y)$ is continuous on $(0,\infty) \times [0,1] \times [0,1]$. This argument also has the side effect of proving the first Feller property $Q_t: C[0,1]\to C[0,1]$.
\par 
Next, it is immediate from \eqref{eq:Bessel-Fourier_coeffs} applied with $i =1$ that $Q_t(x,y)dy$ integrates to 1 and thus is a probability measure. Furthermore for $x,y \not\in \{0,1\}$, the following computation shows $Q_t$ satisfies the Chapman-Kolmogorov equation:
\begin{align*} (Q_{t}Q_s)(x,y)dy
&= \exp\left(\frac{j_{1,\alpha}^2(t+s)}{2}\right)\frac{h_1(y)}{h_1(x)} \int_0^1 R_t(x,z)R_s(z,y)dzdy\\
&= \exp\left(\frac{j_{1,\alpha}^2(t+s)}{2}\right)\frac{h_1(y)}{h_1(x)} \int_0^1 R_t(x,z)\mathbb{P}^x(Y_{t+s} \in dy, \tau > t+s | \tau > t, Y_t = z)dz\\
&= \exp\left(\frac{j_{1,\alpha}^2(t+s)}{2}\right)\frac{h_1(y)}{h_1(x)} \int_0^1 \mathbb{P}^x(Y_{t+s} \in dy, Y_t \in dz, \tau > t+s)\\
&= \exp\left(\frac{j_{1,\alpha}^2(t+s)}{2}\right)\frac{h_1(y)}{h_1(x)} R_{t+s}(x,y)dy\\
&= Q_{t+s}(x,y)dy,
\end{align*}
where we applied the Strong Markov property in the second line. Extending this argument to $x,y \in \{0,1\}$ via continuity shows that $\{Q_t\}_{t \ge 0}$ defines a transition semigroup. To prove this transition semigroup is Feller it now suffices to show $||Q_tf - f|| \to 0$ as $t \to 0$ for all $f \in C[0,1]$. Note 
\begin{equation} \label{eq:Qtf_ev}
Q_tf(x) = \frac{1}{h(x)}\E^x[h_1(Y_{t\wedge \tau})f(Y_{t\wedge \tau})]\exp\left(\frac{j_{1,\alpha^2}}{2}t\right).
\end{equation}
The contribution from the exponential term to $||Q_tf-f||$ is negligible as $t\to 0$ so we will omit it in the following computations. To deal with the remaining terms, it is convenient to work with a subset of $C[0,1]$ to start. Define
\[\mathcal{A} := \{f \in C^2[0,1]  \ : \ f'(0) = f'(1) = 0\}.\]
The condition $f'(1)=0$ is not necessary for the following argument, but we will also employ $\mathcal{A}$ when showing $Q_t$ has generator $\mathcal{L}$ where it will prove useful. For $f \in \mathcal{A}$ Dynkin's formula implies
\begin{equation} \label{eq:Dynkin_Feller_Formula}  \frac{1}{h_1(x)}\mathbb{E}^x[h_1(Y_{t\wedge \tau})f(Y_{t\wedge \tau})] = f(x) + \frac{1}{h_1(x)}\mathbb{E}^x\left[\int_0^{t \wedge \tau} \mathcal{L}_0(h_1f)(Y_s)ds\right]. \end{equation}
Note 
\begin{equation} \label{eq:L0hf}
\mathcal{L}_0(h_1f)(x) = \frac{1}{2}h_1(x)f''(x) + \left(\frac{h_1(x)}{2x} + j_{1,\alpha}x^{-\alpha}J_{\alpha}'(j_{1,\alpha}x)\right)f'(x) - \frac{j_{1,\alpha^2}}{2}h_1(x)f(x) 
\end{equation}
is bounded on $[0,1]$ and thus there exists a constant $C$ such that
\begin{equation} \label{eq:Qtf_expectation_bound}
\left|\frac{1}{h_1(x)}\mathbb{E}^x[h_1(Y_{t\wedge \tau})f(Y_{t\wedge \tau})] - f(x)\right| \le \frac{C}{h_1(x)}\E^x[\tau\wedge t].
\end{equation}
For any $\epsilon > 0$, crudely bounding the expectation by $t$ is enough to deduce $Q_tf(x) \to f(x)$ uniformly for $x \in [0,1-\epsilon]$. To improve this bound near $x=1$, observe that the SDE \eqref{eq:Bessel_SDE} defining a Bessel process allows one to couple a Brownian motion $B_t$ with our Bessel process $Y_t$ in such a way that $B_0=Y_0=x$ and $Y_t \ge B_t$. If $T_{1-x}$ is the time it takes $B_t$ to hit 1, then $T_{1-x} \ge \tau$. The random variable $T_{1-x}$ is well-studied and has density (see e.g. \cite{BS})
\[\frac{1-x}{\sqrt{2\pi s^3}}\exp\left(-\frac{(1-x)^2}{2s}\right)\]
from which one gets
\[\frac{C}{h_1(x)} \E^x[\tau \wedge t] \le \frac{C}{h_1(x)} \E^x[T_{1-x} \wedge t] = \frac{C}{h_1(x)}O((1-x)\sqrt{t}),\]
where the bound holds uniformly as $t\to 0,x\to 1$.
As $h_1(x)$ has a simple root at $x = 1$, $||Q_tf - f|| \to 0$ follows for $f \in \mathcal{A}$ from \eqref{eq:Qtf_expectation_bound}. For general $g \in C[0,1]$, the Stone-Weierstrass theorem implies that $\mathcal{A}$ is dense in $C[0,1]$. Therefore, for $\epsilon > 0$, we can select $f \in \mathcal{A}$ with $||f-g|| < \epsilon$. One then gets
\[||Q_tg-g|| \le ||Q_tf-f|| + ||f-g|| + |||Q_tf-Q_tg|| \le ||Q_tf-f|| + 2\epsilon,\]
where we use the fact that $Q_t$ is contractive. Taking $\epsilon \to 0$ completes the proof that $Q_t$ is Feller.
\par 
We now move on to showing that $Q_t$ has generator $\mathcal{L}$. Take $f \in \mathcal{A}$ as before. Substituting our formulas \eqref{eq:Doob_h_transform}, \eqref{eq:Qtf_ev} for $\mathcal{L}$ and $Q_tf$ into 
 Dynkin's formula \eqref{eq:Dynkin_Feller_Formula}:
\begin{align*}
Q_tf(x) &= \exp\left(\frac{j_{1,\alpha}^2}{2}\right)\left(f(x) + \frac{1}{h_1(x)}\E^x\left[\int_0^{t\wedge \tau} h_1(Y_s)\left(\mathcal{L}f(Y_s)-\frac{j_{1,\alpha}^2}{2}f(Y_s)\right) ds\right]\right). \\
\intertext{Applying $f'(0)=f'(1)=0$ to the formula for $\mathcal{L}_0$ \eqref{eq:L0hf}, one sees that the integrand is bounded and equal to zero at time $\tau$. Hence,}
Q_tf(x)&= \exp\left(\frac{j_{1,\alpha}^2}{2}t\right)\left(f(x) + \int_0^{t} \frac{1}{h_1(x)} \E^x\left[h_1(Y_{s\wedge \tau})\left(\mathcal{L}f(Y_{s\wedge \tau})-\frac{j_{1,\alpha}^2}{2}f(Y_{s\wedge \tau})\right) \right]ds\right) \\
&= \exp\left(\frac{j_{1,\alpha}^2}{2}t\right)\left(f(x) + \int_0^{t} \left(Q_s\mathcal{L}f(x)-\frac{j_{1,\alpha}^2}{2}Q_sf(x)\right)ds\right)
\end{align*}
from which we get
\[\lim_{t\to 0} \frac{Q_tf(x) - f(x)}{t} = \mathcal{L}f(x),\]
as desired.
\par 
Between Proposition \ref{prop:transition_probabilities} and our work thus far in this proof, we have shown that $X^{(n)}_t$ converges in distribution to $X_t$. Extending this to convergence in the sense of finite dimensional distributions is relatively straightforward and essentially amounts to showing $X^{(n)}$ becomes Markovian in the limit. Indeed, let $t_0 = 0 < t_1 < t_2 < \cdots < t_k = t$. Then
\begin{align*}
\mathbb{P}^x(X_t^{(n)} \in dy |X_{t_i}^{(n)} = y_i \text{ for } i \le k-1) &= \mathbb{P}^x(Y_t \in dy| Y_{t_i} = y_{i} \text{ for } i \le k-1, \tau > n)\\
&= \frac{\mathbb{P}^x(Y_t \in dy, \tau > n| Y_{t_i} = y_{i} \text{ for } i \le k-1, \tau > t_{k-1})}{\mathbb{P}^x(\tau > n| Y_{t_i} = y_{i} \text{ for } i \le k-1, \tau > t_{k-1})}
\intertext{Applying the Strong Markov property,}
\mathbb{P}^x(X_t^{(n)} \in dy |X_{t_i}^{(n)} = y_i, \text{ for } i \le k-1) &= \frac{\mathbb{P}^{y_{k-1}}(Y_{t-t_{k-1}} \in dy, \tau > n-t_{k-1})}{\mathbb{P}^{y_{k-1}}(\tau > n-t_{k-1})} \\
&= \mathbb{P}^{y_{k-1}}(Y_{t-t_{k-1}} \in dy | \tau > n - t_{k-1}) \\
&\to Q_{t-t_{k-1}}(y_{k-1},y).
\end{align*}
from which convergence in the sense of finite dimensional distribution follows.
\end{proof}

\subsection*{Acknowledgements} The author would like to thank Alexei Borodin for helpful feedback and bringing the similarity between Ferrari-Spohn diffusions and \cite{GK} to his attention. The author was supported by the NSF Graduate Research fellowship under grant $\#$1745302.

\printbibliography
\end{document}